\documentclass{article}
\usepackage[utf8]{inputenc}

\pdfoutput=1
\usepackage{amsmath}
\usepackage{amsfonts}
\usepackage{amssymb}
\usepackage{centernot}
\usepackage{graphicx}
\usepackage{amsthm}
\usepackage{enumerate}

\usepackage{tikz}
\usetikzlibrary{decorations.markings}
\usetikzlibrary{arrows}
\providecommand{\U}[1]{\protect\rule{.1in}{.1in}}
\newtheorem{theorem}{Theorem}

\newtheorem{corollary}[theorem]{Corollary}

\newtheorem{definition}[theorem]{Definition}

\newtheorem{lemma}[theorem]{Lemma}

\newtheorem{proposition}[theorem]{Proposition}

\theoremstyle{remark}
\newtheorem{remark}{Remark}

\begin{document}

\title{Multivariate Alexander quandles, I. The module sequence of a link}
\author{Lorenzo Traldi\\Lafayette College\\Easton, PA 18042, USA\\traldil@lafayette.edu
}
\date{ }
\maketitle

\begin{abstract}
The multivariate Alexander module of a link $L$ has several subsets that admit quandle operations defined using the module operations. One of them, the fundamental multivariate Alexander quandle, determines the link module sequence of $L$.


\end{abstract}

\section{Introduction}

A classical link in $\mathbb{S}^3$ is a union $L=K_1 \cup \dots \cup K_{\mu}$ of finitely many, pairwise disjoint embedded copies of $\mathbb{S}^1$. All our links are \emph{tame}; that is, the embedding of each component knot $K_i$ is piecewise smooth. Also, all our links are oriented. A \emph{regular projection} of a link $L$ in the plane has only finitely many singularities, all of which are transverse double points called \emph{crossings}. A \emph{link diagram} is obtained from a regular projection by removing two short segments from the underpassing arc at each crossing. The set of crossings of a diagram $D$ is denoted $C(D)$, and the set of connected components of $D$ is denoted $A(D)$; we refer to the elements of $A(D)$ as \emph{arcs} of $D$. The component function $\kappa:A(D) \to \{1,\dots,\mu\}$ is defined so that each arc $a \in A(D)$ is part of the image of $K_{\kappa(a)}$ in $D$.

Two links are \emph{equivalent} or \emph{of the same type} if there is an orientation-preserving autohomeomorphism of $\mathbb{S}^3$ that maps one link onto the other, preserving the link orientations. Equivalent links are represented by diagrams that are related to each other through the Reidemeister moves \cite{R}. 

The purpose of this paper is to provide a new connection between two kinds of structures that  provide link invariants, modules and quandles. Before describing this connection we recall some basic properties of these structures.

Quandles were introduced independently by Joyce \cite{J} and Matveev \cite{M}, and several different modifications of the idea have proven useful in knot theory. A comprehensive, elementary presentation is provided by Elhamdadi and Nelson \cite{EN}. 

\begin{definition}
A \emph{quandle} on a set $Q$ is specified by two binary operations, $\triangleright$ and $\triangleright^{-1}$, which satisfy the following formulas for all elements $x,y,z \in Q$.
\begin{align*}
& \mathrm{(Q1)} \quad x \triangleright x = x\\
& \mathrm{(Q2)} \quad (x \triangleright y) \triangleright^{-1} y =x = (x \triangleright^{-1} y) \triangleright y   \\
&  \mathrm{(Q3)} \quad (x \triangleright y) \triangleright z=(x \triangleright z) \triangleright (y \triangleright z)  
\end{align*}
\end{definition}

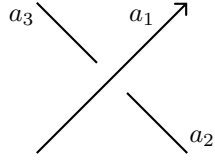
\begin{figure} [bht]
\centering
\begin{tikzpicture} [>=angle 90]
\draw [thick] [<-] (1,1) -- (-1,-1);
\draw [thick] (-1,1) -- (-.2,0.2);
\draw [thick] (0.2,-0.2) -- (1,-1);
\node at (0.4,0.8) {$a_1$};
\node at (-1.2,0.8) {$a_3$};
\node at (1.2,-0.8) {$a_2$};
\end{tikzpicture}
\caption{A crossing.}
\label{crossfig}
\end{figure}

If $D$ is a diagram of a link $L$ then the \emph{fundamental quandle} $Q(L)$ is the quandle generated by elements of $A(D)$, subject to the requirement that at each crossing of $D$ as indicated in Fig.\ \ref{crossfig}, the relations 
$a_2 \triangleright a_1 = a_3$ and $a_3 \triangleright^{-1} a_1 = a_2$ hold. (Note that $a_2$ is on the right side of $a_1$, and $a_3$ is on the left.) Elementary arguments using the Reidemeister moves show that different diagrams of the same link type give rise to isomorphic fundamental quandles \cite{EN, J, M}.

If $L=K_1 \cup \dots \cup K_{\mu}$ then the Alexander module $M_A(L)$ is a module over the ring $\Lambda_{\mu}=\mathbb{Z}[t_1^{\pm1},\dots,t_{\mu}^{\pm1}]$ of Laurent polynomials in the variables $t_1,\dots,t_{\mu}$, with integer coefficients. There are several ways to describe $M_A(L)$. One of the most familiar is the following presentation using generators and relations corresponding to the arcs and crossings of a diagram. 

Let $D$ be a diagram of an oriented link $L$. Let $\Lambda_{\mu}^{A(D)}$ and $\Lambda_{\mu}^{C(D)}$ be the free $\Lambda_{\mu}$-modules on the sets $A(D)$ and $C(D)$, and let $\rho_D:\Lambda_{\mu}^{C(D)} \to \Lambda_{\mu}^{A(D)}$ be the $\Lambda_{\mu}$-linear map with 
\[
\rho_D(c)=(1-t_{\kappa(a_2)})a_1+t_{\kappa(a_1)}a_2-a_3
\]
whenever $c \in C(D)$ is a crossing of $D$ as indicated in Fig.\ \ref{crossfig}. 

\begin{definition}
\label{almod}
If $D$ is a diagram of $L$,
then the \emph{Alexander module} $M_A(L)$ is the cokernel of $\rho_D$. That is, there is a surjection $\gamma_D:\Lambda_{\mu}^{A(D)} \to M_A(L)$ whose kernel is the image of $\rho_D$.
\end{definition}

The maps $\rho_D$ and $\gamma_D$ (named for ``relators'' and ``generators'') vary from one diagram to another, but up to isomorphism, $M_A(L)$ remains the same. An elementary proof of the invariance of $M_A(L)$ under the Reidemeister moves is given in Sec.\ \ref{secdef}.

The \emph{augmentation ideal} $I_{\mu}$  is the ideal of $\Lambda_{\mu}$ generated by $t_1-1,\dots,t_{\mu}-1$. If $L$ is any $\mu$-component link then there is a $\Lambda_{\mu}$-linear epimorphism $\phi_L:M_A(L) \to I_{\mu}$; if $D$ is a diagram of $L$ then $\phi_L(\gamma_D(a))=t_{\kappa(a)}-1$ $ \forall a \in A(D)$. This epimorphism appears in the \emph{link module sequence}
\begin{equation*}
0 \to \ker \phi_L \xrightarrow{\psi} M_A(L) \xrightarrow{\phi_L} I_{\mu} \to 0 \text{,}
\end{equation*}
where $\psi$ is inclusion. The module $\ker \phi_L$ is often called the \emph{Alexander invariant} of $L$. Link module sequences were introduced by Crowell \cite{C1, C2a, C2, C3, CS}. We refer to Crowell's papers and Hillman's book \cite[Chap.\ 4]{H} for a thorough discussion.

%, but we take a moment to recall some of their basic properties.  It can be described without referring to $\phi_L$, as the abelianization $G'/G''$ of the commutator subgroup of the link group $G=\pi_1(\mathbb{S}^3-L)$, with scalar multiplication defined using conjugation in $G$. In general, the relationship between the Alexander invariant and the Alexander module is not clear. For instance, we do not know whether it is possible for two links to have isomorphic Alexander modules and non-isomorphic Alexander invariants. The situation is simpler when $\mu=1$: $I_1$ is a free $\Lambda_1$-module, so the sequence (\ref{sequence}) must split, and hence $M_A(L) \cong I_1 \oplus \ker \phi_L$. 

The usual idea of an isomorphism of exact sequences involves isomorphisms between the modules that appear in the sequences, which combine to form a commutative diagram. For link module sequences, Crowell \cite{C1} required that the isomorphism between the two copies of $I_{\mu}$ be the identity map. 

\begin{definition}
\label{Crowelleq} (\cite{C1})
If $L$ and $L'$ are two links with the same number of components, then the link module sequences of $L$ and $L'$ are \emph{equivalent} if there is a $\Lambda_{\mu}$-linear isomorphism $f:M_A(L) \to M_A(L')$ such that $\phi_L=\phi_{L'} f:M_A(L) \to I_{\mu}$.
\end{definition}
Note that it is not necessary to explicitly require an isomorphism between $\ker \phi_L$ and $\ker \phi_{L'}$; such an isomorphism can be obtained from $f$ by restriction.

\begin{definition}
Let $\Lambda=\mathbb{Z}[t^{\pm1}]$ be the ring of Laurent polynomials in the variable $t$, with integer coefficients. If $L$ is a link with a diagram $D$ then the \emph{reduced Alexander module} $M_A^{red}(L)$ is the $\Lambda$-module obtained by setting all $t_i=t$ in Definition \ref{almod}. The resulting surjection $\Lambda^{A(D)} \to M_A^{red}(L)$ is denoted $\gamma_D^{red}$. 
\end{definition}

If $D$ is a diagram of a link $L$ and $I$ is the ideal of $\Lambda$ generated by $t-1$ then there is a reduced link module sequence,
\begin{equation*}
0 \to \ker \phi_L^{red} \xrightarrow{\psi} M_A^{red}(L) \xrightarrow{\phi_L^{red}} I \to 0 \text{,}
\end{equation*}
with $\phi_L^{red}(\gamma_D^{red}(a))=t-1$ $\forall a \in A(D)$. The module $\ker \phi_L^{red}$ is the \emph {reduced Alexander invariant} of $L$. The principal ideal $I$ is a projective $\Lambda$-module -- it is isomorphic to $\Lambda$ -- so all reduced link module sequences split. Consequently there is little reason to work with both the reduced Alexander module and the reduced Alexander invariant; most references focus on one or the other.

It is an unsurprising fact that when we pass to reduced Alexander modules, setting all $t_i=t$ entails a significant loss of information. 

\begin{proposition}
\label{red}
The link module sequence is a strictly stronger invariant than the reduced link module sequence. That is:
\begin{enumerate}
    \item If two links have equivalent link module sequences, then they have equivalent reduced link module sequences.
    \item If two links have equivalent reduced link module sequences, then they might not have equivalent link module sequences.
\end{enumerate}
\end{proposition}

We guess that Proposition \ref{red}  is known, but we do not have a reference. A proof is given in Section \ref{secred}.

The traditional way to associate quandles to Alexander modules involves applying the following easily verified observation to either the reduced Alexander module or the reduced Alexander invariant \cite{EN, J, M}.

\begin{proposition}
\label{standardq}
If $M$ is a $\Lambda$-module, then there is a quandle structure on $M$ given by $x \triangleright y=tx+(1-t)y$ and $x \triangleright^{-1} y=t^{-1} \cdot (x+(t-1)y)$.
\end{proposition}

The quandles described in Proposition \ref{standardq} are called \emph{Alexander quandles} in the literature. We refer to them as \emph{standard} Alexander quandles, to avoid confusion with the quandles described in Definition \ref{multiq} below.

\begin{proposition}
\label{multiq1}
Suppose $L$ is a link, and $\phi_L:M_A(L) \to I_{\mu}$ is the map appearing in the link module sequence of $L$. Then the quandle axioms \textup{(Q1)} and \textup{(Q3)} are satisfied by the operation on $M_A(L)$ defined by the formula \[x \triangleright y=(\phi_L(y)+1)x-\phi_L(x)y.\] 
\end{proposition}

\begin{proposition}
\label{multiq2} 
Let $L$ be a link, and let 
\[
U(L) = \{ x \in M_A(L) \mid \phi_L(x)+1 \text{ is a unit of } \Lambda_{\mu} \}.
\]
Then the operation $\triangleright$ of Proposition \ref{multiq1} defines a quandle structure on $U(L)$, with
\[x \triangleright^{-1} y=(\phi_L(y)+1)^{-1} \cdot (x+\phi_L(x)y).
\]
\end{proposition}

\begin{definition}
\label{multiq}
We call $U(L)$ the \emph{total multivariate Alexander quandle} of $L$. If $D$ is a diagram of $L$ then the \emph{fundamental multivariate Alexander quandle} $Q_A(L)$ is the minimal subquandle of $U(L)$ with $\gamma_D(a) \in Q_A(L)$ $\forall a \in A(D)$.
\end{definition}

Here are two remarks about Definition \ref{multiq}.

\begin{remark}
\label{rem1}
 The $\triangleright$ and $\triangleright^{-1}$ operations of Propositions \ref{multiq1} and \ref{multiq2} are closely connected to the maps $\rho_D$ and $\gamma_D$ of Definition \ref{almod}. If $c$ is a crossing with arcs indexed as in Fig.\ \ref{crossfig}, then $\gamma_D \rho_D (c)=0$ and $\kappa(a_2)=\kappa(a_3)$, so
\begin{align*}
& \gamma_D(a_2) \triangleright \gamma_D(a_1) = \gamma_D(a_2) \triangleright \gamma_D(a_1) - \gamma_D(\rho_D(c))\\
& =(\phi_L(\gamma_D(a_1))+1)\gamma_D(a_2)-\phi_L(\gamma_D(a_2))\gamma_D(a_1)-(1-t_{\kappa(a_2)})\gamma_D(a_1)
\\
& \qquad - t_{\kappa(a_1)}\gamma_D(a_2) + \gamma_D(a_3)
\\
& =t_{\kappa(a_1)}\gamma_D(a_2)-(t_{\kappa(a_2)}-1)\gamma_D(a_1)-(1-t_{\kappa(a_2)})\gamma_D(a_1)\\
& \qquad - t_{\kappa(a_1)}\gamma_D(a_2) + \gamma_D(a_3) 
\\ & = \gamma_D(a_3) \text{,}
\end{align*}
\begin{center} and \end{center}
\begin{align*}
& \gamma_D(a_3) \triangleright^{-1} \gamma_D(a_1) = \gamma_D(a_3) \triangleright^{-1} \gamma_D(a_1) + t_{\kappa(a_1)}^{-1} \gamma_D(\rho_D(c))\\
& =(\phi_L(\gamma_D(a_1))+1)^{-1} \cdot (\gamma_D(a_3)+\phi_L(\gamma_D(a_3))\gamma_D(a_1)) + t_{\kappa(a_1)}^{-1} \gamma_D(\rho_D(c))\\
& =t_{\kappa(a_1)}^{-1} \cdot(\gamma_D(a_3)+(t_{\kappa(a_3)}-1)\gamma_D(a_1)) + t_{\kappa(a_1)}^{-1} (1-t_{\kappa(a_2)})\gamma_D(a_1)\\
& \qquad+ t_{\kappa(a_1)}^{-1}t_{\kappa(a_1)}\gamma_D(a_2) - t_{\kappa(a_1)}^{-1}\gamma_D(a_3)
\\ & = \gamma_D(a_2) \text{.}
\end{align*}
\end{remark}

\begin{remark}
\label{rem2}
%$U(L)$ is not a submodule of $M_A(L)$. For one thing, it is not closed under multiplication by $-1$: If $x \in U(L)$ and $\phi_L(x) \neq 0$, then $\phi_L(x)+1= \pm t_1^{n_1}\cdots t_{\mu}^{n_{\mu}}$ for some integers $n_1,\dots,n_{\mu}$, not all of which equal $0$. But then $\phi_L(-x)+1 = 2 \mp t_1^{n_1}\cdots t_{\mu}^{n_{\mu}}$, which is not a unit of $\Lambda_{\mu}$. It follows that $U(L)$ is not a standard Alexander quandle, and neither is any subquandle that is not contained in $\ker \phi_L$.
In general, the quandles $U(L)$ and $Q_A(L)$ are not standard Alexander quandles. For example, suppose $D$ is a crossing-free diagram of the three-component unlink $L^u_3=K_1 \cup K_2 \cup K_3$. Let $A(D)=\{a_1,a_2,a_3\}$, with $\kappa(a_i)=i$ for each $i$. Then $\gamma_D(a_i) \triangleright \gamma_D(a_j)=t_j\gamma_D(a_i) - (t_i-1)\gamma_D(a_j)$ for $1 \leq i,j \leq 3$, so we have
\begin{align*}
& \gamma_D(a_1) \triangleright (\gamma_D(a_2) \triangleright \gamma_D(a_3))= \gamma_D(a_1) \triangleright (t_3\gamma_D(a_2) - (t_2-1)\gamma_D(a_3))=\\
& (t_3(t_2-1)-(t_2-1)(t_3-1)+1)\gamma_D(a_1)-(t_1-1)(t_3\gamma_D(a_2) - (t_2-1)\gamma_D(a_3))
\\
& =t_2\gamma_D(a_1)-(t_1-1)t_3\gamma_D(a_2)+(t_1-1)(t_2-1)\gamma_D(a_3)\text{,}
\end{align*}
\begin{center} and \end{center}
\begin{align*}
& (\gamma_D(a_1) \triangleright \gamma_D(a_2)) \triangleright (\gamma_D(a_1) \triangleright \gamma_D(a_3)) \\
& = (t_2\gamma_D(a_1) - (t_1-1)\gamma_D(a_2)) \triangleright (t_3\gamma_D(a_1) - (t_1-1)\gamma_D(a_3))\\
& =(t_3(t_1-1)-(t_1-1)(t_3-1)+1)(t_2\gamma_D(a_1) - (t_1-1)\gamma_D(a_2))\\
& -(t_2(t_1-1)-(t_1-1)(t_2-1))(t_3\gamma_D(a_1) - (t_1-1)\gamma_D(a_3))
\\
& =t_1(t_2\gamma_D(a_1) - (t_1-1)\gamma_D(a_2)) -(t_1-1)(t_3\gamma_D(a_1) - (t_1-1)\gamma_D(a_3))
\\
& =(t_1t_2-t_1t_3+t_3)\gamma_D(a_1)-t_1(t_1-1)\gamma_D(a_2)+(t_1-1)^2\gamma_D(a_3).
\end{align*}
As $C(D)= \emptyset$, $\gamma_D:\Lambda_3^{A(D)} \to M_A(L^u_3)$ is an isomorphism. Therefore the calculations above imply that $\gamma_D(a_1) \triangleright (\gamma_D(a_2) \triangleright \gamma_D(a_3)) \neq (\gamma_D(a_1) \triangleright \gamma_D(a_2)) \triangleright (\gamma_D(a_1) \triangleright \gamma_D(a_3))$. However the $\triangleright$ formula of Proposition \ref{standardq} implies that all standard Alexander quandles have $x \triangleright (y \triangleright z) = (x \triangleright y) \triangleright (x \triangleright z)$ $\forall x,y,z$. It follows that neither $Q_A(L^u_3)$ nor $U(L^u_3)$ is a standard Alexander quandle.
\end{remark}

The quandles $U(L)$ and $Q_A(L)$ are link invariants:

\begin{proposition}
\label{qinv}
If $D$ and $D'$ are diagrams of the same link type, then the resulting $U(L)$ quandles are isomorphic, and  the resulting $Q_A(L)$ quandles are isomorphic.
\end{proposition}

Notice that even though the component indices in $L$ are used to define $\phi_L$, the quandle structures of $U(L)$ and $Q_A(L)$ do not reflect the indices explicitly. That is, if $L=K_1 \cup \dots \cup K_{\mu}$ and $L'=K'_1 \cup \dots \cup K'_{\mu}$ are the same except for the indexing of their components, then $U(L) \cong U(L')$ and $Q_A(L) \cong Q_A(L')$. In contrast, the component indices are reflected in the elements of the coefficient ring $\Lambda_{\mu}$, and (hence) in the modules that appear in the link module sequence. Our main theorem shows that aside from this lack of sensitivity to component indices, $Q_A(L)$ is at least as strong an invariant as the link module sequence.

\begin{theorem}
\label{main}
Consider the following relations that might hold between invariants of two links $L=K_1 \cup \dots \cup K_{\mu}$ and $L'=K'_1 \cup \dots \cup K'_{\mu'}$. 
\begin{enumerate}
\item The fundamental quandles of $L$ and $L'$ are isomorphic.
\item The fundamental multivariate Alexander quandles of $L$ and $L'$ are isomorphic.
\item The components of $L$ and $L'$ can be indexed so that the link module sequences of $L$ and $L'$ are equivalent.
\item The total multivariate Alexander quandles of $L$ and $L'$ are isomorphic.
\end{enumerate}
Then the implications $1 \implies 2 \implies 3 \implies 4$ hold.
\end{theorem}

At least two of the implications of Theorem \ref{main} have false converses: $2 \centernot \implies 1$ and $4 \centernot \implies 3$. At the time of writing, we do not know whether $3 \implies 2$ is valid.\footnote{Several months after the present paper was completed, we were able to show that $3 \centernot \implies 2$ in Theorem \ref{main}. See the third paper in this series \cite{mvaq3} for details.}

Proposition \ref{red} and Theorem \ref{main} give us the following.
\begin{corollary}
\label{cor}
Multivariate Alexander quandles associated with multivariate Alexander modules are strictly stronger link invariants than standard Alexander quandles associated with reduced Alexander invariants or reduced Alexander modules.
\end{corollary}

Here is an outline of the paper. Propositions \ref{multiq1} and \ref{multiq2} are proven in Sec.\ \ref{secq}, and Proposition \ref{qinv} is proven in Sec.\ \ref{secdef}. In Sec.\ \ref{seclemm} we lay the groundwork for the proof of Theorem \ref{main} by constructing a presentation of the module $M_A(L)$ that uses only information from the fundamental quandle $Q(L)$, and a related presentation that uses only information from $Q_A(L)$. The proof of Theorem \ref{main} is completed in Sec.\ \ref{secproof}, and counterexamples to the converses of the implications $1 \implies 2$ and $3 \implies 4$ are given in Sec.\ \ref{seccoun}. Proposition \ref{red} is proven in Sec.\ \ref{secred}. In Sec.\ \ref{secprob} we discuss possible further developments involving the ideas of the paper.

\section{Propositions \ref{multiq1} and \ref{multiq2}}
\label{secq}

These two propositions hold in a more general setting. Suppose $R$ is a commutative ring with unity, $M$ is an $R$-module, and $\phi:M \to R$ is an $R$-linear map. For $x,y \in M$, let $x \triangleright y=(\phi(y)+1)x-\phi(x)y$.  Then the first and third quandle axioms hold:

\medskip

(Q1) If $x \in M$, then $x \triangleright x = (\phi(x)+1)x-\phi(x)x =x$. 

\medskip

(Q3) If $x,y,z \in M$, then 
\begin{align*}
& (x \triangleright y) \triangleright z=(\phi(z)+1)(x \triangleright y)-\phi(x \triangleright y)z \\
& =(\phi(z)+1)((\phi(y)+1)x-\phi(x)y)-\phi((\phi(y)+1)x-\phi(x)y)z \\
& =(\phi(z)+1)(\phi(y)+1)x+(\phi(z)+1)(-\phi(x))y-((\phi(y)+1)\phi(x) -\phi(x)\phi(y))z \\
& =(\phi(z)+1)(\phi(y)+1)x-(\phi(z)+1)\phi(x)y-\phi(x)z \text{,}
\end{align*}
\begin{align*}
& \qquad \qquad \qquad \qquad \qquad \qquad \text{and} \\
& (x \triangleright z) \triangleright (y \triangleright z)=((\phi(z)+1)x-\phi(x)z) \triangleright ((\phi(z)+1)y-\phi(y)z) \\
& =(\phi((\phi(z)+1)y-\phi(y)z)+1)((\phi(z)+1)x-\phi(x)z) \\
& -\phi((\phi(z)+1)x-\phi(x)z)((\phi(z)+1)y-\phi(y)z) \\
& =(\phi(y)+1)((\phi(z)+1)x-\phi(x)z)-\phi(x)((\phi(z)+1)y-\phi(y)z) \\
& =(\phi(y)+1)(\phi(z)+1)x+\phi(x)(-\phi(z)-1)y+(-\phi(x)(\phi(y)+1)+\phi(x)\phi(y))z \\
& =(\phi(y)+1)(\phi(z)+1)x-\phi(x)(\phi(z)+1)y-\phi(x)z=(x \triangleright y) \triangleright z.
\end{align*}

Now let $U(\phi)=\{x \in M \mid \phi(x)+1 \text{ is a unit of } R\}$, and for $x,y \in U(\phi)$ let $x \triangleright^{-1} y=(\phi(y)+1)^{-1} \cdot (x+\phi(x)y).$ As $\phi$ is $R$-linear, 
\begin{align*}
& \phi(x \triangleright y)=\phi((\phi(y)+1)x-\phi(x)y)
=(\phi(y)+1)\phi(x)-\phi(x)\phi(y)=\phi(x) \text { and}\\
& \phi(x \triangleright^{-1} y)=(\phi(y)+1)^{-1} \phi(x+\phi(x)y) =(\phi(y)+1)^{-1}\phi(x)(1+\phi(y))=\phi(x).
\end{align*}
It follows that $U(\phi)$ is closed under both $\triangleright$ and $\triangleright^{-1}$. The proof is completed by verifying the second quandle axiom:

\medskip

(Q2) If $x,y \in U(\phi)$, then 
\begin{align*}
& (x \triangleright y) \triangleright^{-1} y=(\phi(y)+1)^{-1} \cdot  (x \triangleright y + \phi(x \triangleright y)y) \\
& =(\phi(y)+1)^{-1} \cdot  ((\phi(y)+1)x-\phi(x)y + \phi((\phi(y)+1)x-\phi(x)y)y) \\
& =(\phi(y)+1)^{-1} \cdot  ((\phi(y)+1)x+(-\phi(x) + (\phi(y)+1)\phi(x)-\phi(x)\phi(y))y) \\
& =(\phi(y)+1)^{-1} \cdot  ((\phi(y)+1)x+0y)=x
\end{align*}
\begin{center} and \end{center}
\begin{align*}
& (x \triangleright^{-1} y) \triangleright y= (\phi(y)+1)(x \triangleright^{-1} y)-\phi(x \triangleright^{-1} y)y \\
& =(\phi(y)+1)(\phi(y)+1)^{-1} \cdot(x + \phi(x)y) -\phi((\phi(y)+1)^{-1} \cdot (x+\phi(x)y))y \\
& =x + \phi(x)y - (\phi(y)+1)^{-1} \cdot (\phi(x)+\phi(x)\phi(y))y \\
& =x + \phi(x)y - (\phi(y)+1)^{-1} \cdot (1+\phi(y))\phi(x)y = x \text{.}
\end{align*}
\qed
%This more general construction does not seem to be as useful in knot theory as Proposition \ref{multiq}; restricting to the particular elements with $\phi_L(x)+1=t_i$ allows $Q_A(L)$ to reflect more precise information about the peripheral structure of the Alexander module.

\section{Invariance under the Reidemeister moves}
\label{secdef}

In this section we prove that different diagrams of the same link type provide isomorphic multivariate Alexander quandles $U(L)$ and $Q_A(L)$. Along the way we also verify the well-known fact that different diagrams give rise to link module sequences that are equivalent, in the sense of Definition \ref{Crowelleq}.

Let $D$ be a diagram of a $\mu$-component link $L$. To allow for hypothetical dependence on diagrams, we temporarily modify the notation used in the introduction: $M_A(D)$ replaces $M_A(L)$, $U(D)$ replaces $U(L)$, and so on.  Let $\Phi_D:\Lambda_{\mu}^{A(D)} \to I_{\mu}$ be the $\Lambda_{\mu}$-linear map with $\Phi_D(a)=t_{\kappa(a)}-1$ for each $a \in A(D)$. If $c$ is a crossing of $D$ with arcs indexed as in Fig.\ \ref{crossfig}, then $\kappa(a_3)=\kappa(a_2)$. It follows that the image of the $c$ generator of $\Lambda_{\mu}^{C(D)}$ under the composition $\Phi_D \rho_D$ is
\begin{gather*}
\Phi_D( \rho_D(c))=(1-t_{\kappa(a_2)}) \Phi_D(a_1)+t_{\kappa(a_1)}\Phi_D(a_2)-\Phi_D(a_3) \\
= (1-t_{\kappa(a_2)})(t_{\kappa(a_1)}-1)+t_{\kappa(a_1)}(t_{\kappa(a_2)}-1)-(t_{\kappa(a_2)}-1) = 0 \text{,}
\end{gather*}
so $\Phi_D$ defines a $\Lambda_{\mu}$-linear map $\phi_D:M_A(D) \to I_{\mu}$ with $\phi_D (\gamma_D(a))=\Phi_D(a)$ $\forall a \in A(D)$. It is evident that $\Phi_D$ and $\phi_D$ are surjective.

\begin{proposition}
\label{seqinv}
If $D$ and $D'$ are diagrams of the same link type then there is an isomorphism $f:M_A(D) \to M_A(D')$ such that $\phi_D=\phi_{D'} f$. Moreover, $f$ restricts to quandle isomorphisms $U(D) \cong U(D')$ and $Q_A(D) \cong Q_A(D')$.
\end{proposition}
\begin{proof}
It suffices to verify the proposition if $D'$ is obtained from $D$ by applying a single Reidemeister move. We give details for one instance of each type of move, and leave it to the reader to consider other instances, obtained by reversing crossings or orientations.

\begin{figure} [bht]
\centering
\begin{tikzpicture} [>=angle 90]
\draw [thick] [->] (-0.25,-0.25) -- (.75,.75);
%\draw [thick] [<-] (-0.25,-0.25) .. controls (-2,-2) and (-2,2) .. (-0.25,0.25);
\draw [thick] (-0.25,-0.25) to [out=225, in=0] (-1,-.75);
\draw [thick] (-1,-.75) to [out=180, in=-90] (-1.9,0);
\draw [thick] (-1,.75) to [out=180, in=90] (-1.9,0);
\draw [thick] (-1,.75) to [out=0, in=-45] (-0.2,0.2);
\draw [thick] (.75,-0.75) -- (0.1,-0.1);
\draw [thick] [->] (-4,-0.75) -- (-4,.75);
\node at (-2.2,0) {$a_1$};
\node at (1,-0.6) {$a_2$};
\node at (-4.3,0) {$a_1$};
\node at (0,-.45) {$c_0$};
\end{tikzpicture}
\caption{An $\Omega.1$ move.}
\label{firstmove}
\end{figure}
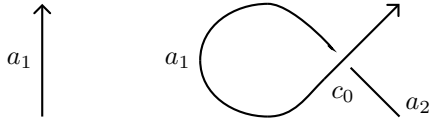

Suppose $D'$ is obtained from $D$ by adding a single new crossing $c_0$ in an $\Omega.1$ move, as illustrated in Fig.\ \ref{firstmove}. As $\kappa(a_1)=\kappa(a_2)$, the image under $\rho_{D'}$ of the $c_0$ generator of $\Lambda_{\mu}^{C(D')}$ is $t_{\kappa(a_1)} \cdot( a_2 - a_1)$. It follows that if $F:\Lambda_{\mu}^{A(D')} \to \Lambda_{\mu}^{A(D)}$ is the $\Lambda_{\mu}$-linear surjection with $F(a_1)=F(a_2)=a_1$ and $F(a)=a$ otherwise, then (i) $F ( \rho_{D'} (c)) = \rho_D (c)$ for every crossing $c \neq c_0$, (ii) $F ( \rho_{D'} (c_0))  = 0$, and (iii) $\Phi_{D'}=\Phi_{D}F$. Hence $F$ induces an isomorphism $f$ between the cokernels of $\rho_{D'}$ and $\rho_{D}$, with $\phi_{D'}=\phi_D f$. The equality $\phi_{D'}=\phi_D f$ implies that $f(U(D'))=U(D)$, and $f$ defines a quandle isomorphism between $U(D')$ and $U(D)$. Notice moreover that $F(A(D'))=A(D) \subset \Lambda_{\mu}^{A(D)}$; it follows that $f\gamma_{D'}(A(D'))=\gamma_D(A(D))$, so $f(Q_A(D'))=Q_A(D)$. 

Now, suppose $D'$ is obtained from $D$ by adding two new crossings in an $\Omega.2$ move, as in Fig.\ \ref{secondmove}. The generators of $\Lambda_{\mu}^{C(D')}$ corresponding to these new crossings are mapped by $\rho_{D'}$ to the elements 
\[
(1-t_{\kappa(a_2)})a_1+t_{\kappa(a_1)}a_2-a_3 \qquad \text{      and      } \qquad (1-t_{\kappa(a_2)})a_1+t_{\kappa(a_1)}a_4-a_3
\]
of $\Lambda_{\mu}^{A(D')}$. It follows that if $F:\Lambda_{\mu}^{A(D')} \to \Lambda_{\mu}^{A(D)}$ is the $\Lambda_{\mu}$-linear surjection with $F(a_3)=(1-t_{\kappa(a_2)})a_1+t_{\kappa(a_1)}a_2$, $F(a_4)=a_2$ and $F(a)=a$ otherwise, then (i) $F ( \rho_{D'} (c)) = \rho_D (c)$ for every crossing $c$ not pictured in Fig.\ \ref{secondmove}, (ii) $F ( \rho_{D'} (c))  = 0$ for both crossings pictured in Fig.\ \ref{secondmove}, and (iii) $\Phi_{D'}=\Phi_{D}F$. Again, we conclude that $F$ induces an isomorphism $f$ between the cokernels of $\rho_{D'}$ and $\rho_{D}$, with $\phi_{D'}=\phi_D f$. And again, $\phi_{D'}=\phi_D f$ implies that $f(U(D))=U(D')$, and $f$ defines a quandle isomorphism between $U(D')$ and $U(D)$. In this case $f\gamma_{D'}(A(D'))$ properly contains $\gamma_D(A(D))$, but the one extra element is $f\gamma_{D'}(a_3)=\gamma_D(a_2) \triangleright \gamma_D(a_1)$, so $f(Q_A(D'))=Q_A(D)$. 

\begin{figure} [bht]
\centering
\begin{tikzpicture} [>=angle 90]
\draw [thick] [<-] (1,-1) to [out=-180, in=-90] (-1,0);
\draw [thick] (1,1) to [out=-180, in=90] (-1,0);
\draw [thick] (-2,-1) -- (-.9,-0.7);
\draw [thick] (-2,1) -- (-.9,0.7);
\draw [thick] (-0.45,-0.55) to [out=20, in=-20] (-0.45,0.55);
\draw [thick] (-6,-1) -- (-6,1);
\draw [thick] [<-] (-5,-1) -- (-5,1);
\node at (1,0.7) {$a_1$};
\node at (-2,0.7) {$a_2$};
\node at (.2,0) {$a_3$};
\node at (-2,-0.7) {$a_4$};
\node at (-6.3,0) {$a_2$};
\node at (-5.3,0) {$a_1$};
\end{tikzpicture}
\caption{An $\Omega.2$ move.}
\label{secondmove}
\end{figure}
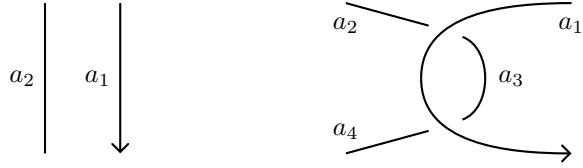

Let $D$ and $D'$ be the diagrams on the left and right of Fig.\ \ref{thirdmove} (respectively). The images under $\rho_D$ of the generators of $\Lambda_{\mu}^{C(D)}$ corresponding to the three pictured crossings of $D$ are 
\begin{align*}
& \rho_1=(1-t_{\kappa(a_2)})a_1+t_{\kappa(a_1)}a_2-a_3 \text{,} \\
& \rho_2=(1-t_{\kappa(a_6)})a_1+t_{\kappa(a_1)}a_6-a_5 \text{, and}\\
& \rho_3=(1-t_{\kappa(a_4)})a_3+t_{\kappa(a_3)}a_5-a_4\text{.}
\end{align*}
The images under $\rho_{D'}$ of the generators of $\Lambda_{\mu}^{C(D')}$ corresponding to the three pictured crossings of $D'$ are 
\begin{align*}
& \rho'_1=(1-t_{\kappa(a_2)})a_1+t_{\kappa(a_1)}a_2-a_3 \text{,} \\
& \rho'_2=(1-t_{\kappa(a_4)})a_1+t_{\kappa(a_1)}a_7-a_4 \text{,  and} \\
& \rho'_3=(1-t_{\kappa(a_6)})a_2+t_{\kappa(a_2)}a_6-a_7\text{.}
\end{align*}

Let $F:\Lambda_{\mu}^{A(D)} \to \Lambda_{\mu}^{A(D')}$ be the $\Lambda_{\mu}$-linear map with $F(a)=a$ for $a \neq a_5$, and  $F(a_5)=a_4+t_{\kappa(a_1)}(a_6-a_7)$. Then the image of $F$ includes every $a \neq a_7 \in A(D')$, because each such $a$ equals $F(a)$, and the image of $F$ also includes $F(t_{\kappa(a_1)}^{-1} \cdot(a_4-a_5+t_{\kappa(a_1)}a_6))=a_7$. It follows that $F$ is surjective. In addition, $F$ is injective: if $x \neq 0 \in \Lambda_{\mu}^{A(D)}$ then either the $a_5$ coordinate of $x$ is nonzero, in which case the $a_7$ coordinate of $F(x)$ is nonzero, or else the $a_5$ coordinate of $x$ is $0$, in which case $x$ and $F(x)$ are precisely the same as linear combinations of generators. We conclude that $F$ is an isomorphism of $\Lambda_{\mu}$-modules.

Notice that $\kappa(a_2)=\kappa(a_3)$ and $\kappa(a_4)=\kappa(a_5)=\kappa(a_6)=\kappa(a_7)$, so we have
\begin{align*}
& F(\rho_1)=\rho'_1, \\
& F(\rho_2)=(1-t_{\kappa(a_6)})a_1+t_{\kappa(a_1)}a_6-a_4-t_{\kappa(a_1)}a_6+t_{\kappa(a_1)}a_7=\rho'_2\text{,  and} \\
& F(\rho_3)=(1-t_{\kappa(a_4)})a_3+t_{\kappa(a_3)}(a_4+t_{\kappa(a_1)}(a_6-a_7))-a_4 \\
& \qquad \quad =(t_{\kappa(a_4)}-1)\rho'_1+(1-t_{\kappa(a_2)})\rho'_2+t_{\kappa(a_1)}\rho'_3.
\end{align*}
It follows that $
F(\rho_D(\Lambda_{\mu}^{C(D)}))=\rho_{D'}(\Lambda_{\mu}^{C(D')})$, so $F$ induces an isomorphism $f$ between the cokernels of $\rho_D$ and $\rho_{D'}$. The equation $\Phi_D=\Phi_{D'} F$ follows directly from the definition of $F$, and it implies that  $\phi_D=\phi_{D'} f$. As in the first two cases, $\phi_{D'}=\phi_D f$ implies that $f(U(D))=U(D')$, and $f$ defines a quandle isomorphism between $U(D)$ and $U(D')$.

\begin{figure} 
\centering
\begin{tikzpicture} [>=angle 90]  
\draw [thick] [->] (2.4,-1.2) -- (-2.4,1.2); 
\draw [thick] [->] (.12,.06) -- (2.4,1.2); 
\draw [thick] (-.12,-.06) -- (-2.4,-1.2); 
\draw [thick] (-2.4,0) to [out=0, in=210] (-1.55,.63);
\draw [thick] (-1.32,.78) to [out=30,in=150] (1.32,.78);
\draw [thick] (2.4,0) to [out=180, in=-30] (1.55,.63);
\draw [thick] [->] (-4.8,-1.2) -- (-9.6,1.2);
\draw [thick] [->] (-7.08,.06) -- (-4.8,1.2);
\draw [thick] (-7.32,-.06) -- (-9.6,-1.2);
\draw [thick] (-9.6,0) to [out=0, in=-210] (-8.75,-.63);
\draw [thick] (-8.52,-.78) to [out=-30,in=-150] (-5.88,-.78);
\draw [thick] (-4.8,0) to [out=180, in=30] (-5.65,-.63);
\node at (-4.98,-.84) {$a_1$};
\node at (-5.4,0) {$a_6$};
\node at (-4.98,.78) {$a_2$};
\node at (-7.2,-.9) {$a_5$};
\node at (-9.48,-.84) {$a_3$};
\node at (-9,0) {$a_4$};
\node at (2.22,-.84) {$a_1$};
\node at (1.8,0) {$a_6$};
\node at (2.22,.78) {$a_2$};
\node at (0,.9) {$a_7$};
\node at (-2.28,-.84) {$a_3$};
\node at (-1.8,0) {$a_4$};
\end{tikzpicture}
\caption{An $\Omega.3$ move.}
\label{thirdmove}
\end{figure}
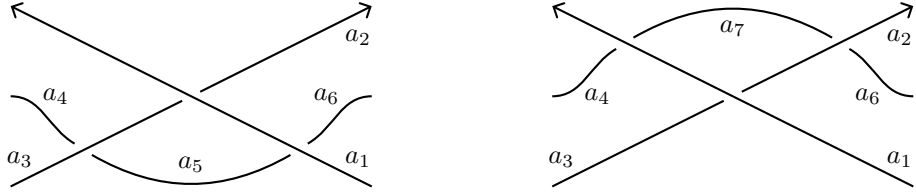

In this case each of the sets $f\gamma_D(A(D)),\gamma_{D'}(A(D'))$ has one element that is not included in the other: $f\gamma_D(A(D))$ contains $f\gamma_D(a_5)$, and $\gamma_{D'}(A(D'))$ contains $\gamma_{D'}(a_7)$. Considering the element $\rho'_3$ of the image of $\rho_{D'}$, we have
\begin{align*}
& \gamma_{D'}(a_7)=(1-t_{\kappa(a_6)})\gamma_{D'}(a_2) +t_{\kappa(a_2)}\gamma_{D'}(a_6) \\
& = \gamma_{D'}(a_6) \triangleright \gamma_{D'}(a_2)= f(\gamma_{D}(a_6) \triangleright \gamma_{D}(a_2)) \text{,}
\end{align*}
so $\gamma_{D'}(a_7) \in f(Q_A(D))$. And considering the element $\rho'_2$ of the image of $\rho_{D'}$, we have
\[
f\gamma_D(a_5)=\gamma_{D'}(a_4+t_{\kappa(a_1)}(a_6-a_7))=\gamma_{D'}(t_{\kappa(a_1)}a_6+(1-t_{\kappa(a_4)})a_1) \text{,}
\]
so since $\kappa(a_4)=\kappa(a_6)$, $f(\gamma_D(a_5))=\gamma_{D'}(a_6) \triangleright \gamma_{D'}(a_1) \in Q_A(D')$. It follows that $f(Q_A(D))=Q_A(D').$
\end{proof}

\section{Quandle-based presentations of $M_A(L)$}
\label{seclemm}

Let $D$ be a diagram of a link $L$. The fundamental quandle $Q(L)$ is obtained by first constructing the set of all words obtained from elements of $A(D)$ using the binary operation symbols $\triangleright$ and $\triangleright^{-1}$, and then modding out by the equivalence relation generated by the quandle axioms (Q1), (Q2) and (Q3), along with the requirement that at each crossing of $D$ as illustrated in Fig.\ \ref{crossfig}, $a_3=a_2 \triangleright a_1$ and $a_2=a_3 \triangleright^{-1} a_1$. Joyce \cite{J} and Matveev \cite{M} proved that up to isomorphism, the fundamental quandle is a link type invariant.

\begin{lemma}
\label{lem12}
 There is a surjective quandle homomorphism $\sigma:Q(L) \to Q_A(L)$, which has $\sigma(a)=\gamma_D(a)$ $\forall a \in A(D)$.
\end{lemma}

\begin{proof}
 The fundamental quandle is the largest quandle generated by $A(D)$ that has $a_3=a_2 \triangleright a_1$ and $a_2=a_3 \triangleright^{-1} a_1$ at each crossing of $D$. The lemma follows, because $Q_A(L)$ is a quandle generated by $\gamma_D(A(D))$ and according to Remark \ref{rem1} of the introduction, $\gamma_D(a_3) = \gamma_D(a_2) \triangleright \gamma_D(a_1)$ and $\gamma_D(a_2) = \gamma_D(a_3) \triangleright^{-1} \gamma_D(a_1)$ at each crossing of $D$.
\end{proof}

Here is a standard definition.

\begin{definition}
An \emph{orbit} in a quandle $Q$ is an equivalence class of elements of $Q$ under the equivalence relation defined by $x \sim x \triangleright y \sim x \triangleright^{-1} y$ $\forall x,y \in Q$.
\end{definition}

\begin{lemma}
\label{lem14}
If $L$ is a link then $Q(L)$ has one orbit for each component of $L$. The same holds for $Q_A(L)$.
\end{lemma}
\begin{proof}
The quandle $Q(L)$ is generated by $A(D)$, so every orbit in $Q(L)$ contains at least one element of $A(D)$. For the same reason, every orbit in $Q_A(L)$ contains at least one element of $\gamma_D(A(D))$.

Suppose $a \in A(D)$ has $\kappa(a)=k$. Every other arc $a'$ with $\kappa(a')=k$ is related to $a$ under $\sim$ in $Q(L)$, because we can construct a sequence of $\triangleright$ and $\triangleright^{-1}$ operations in $Q(L)$ that lead from $a$ to $a'$: we simply walk along $K_k$ from $a$ to $a'$, and apply $\triangleright$ or $\triangleright^{-1}$ each time we pass from one arc to another, as dictated by the orientation of the overpassing arc. The same reasoning shows that a single orbit of $Q_A(L)$ contains both $\gamma_D(a)$ and $\gamma_D(a')$.

It might seem possible that arcs from two different components of $L$ could be related under $\sim$ in $Q(L)$ or $Q_A(L)$. To guarantee that this is not the case, note that $\phi_L$ is constant on each orbit in $Q_A(L)$: if $x,y \in Q_A(L)$ then
\[
\phi_L(x \triangleright y)= \phi_L ((\phi_L(y)+1)x-\phi_L(x)y) = (\phi_L(y)+1)\phi_L(x)-\phi_L(x)\phi_L(y)=\phi_L(x)
\]
\begin{center} and \end{center}
\[
\phi_L(x \triangleright^{-1} y)=\phi_L((\phi_L(y)+1)^{-1} \cdot (x+\phi_L(x)y))
\]
\[
= (\phi_L(y)+1)^{-1}  (\phi_L(x)+\phi_L(x)\phi_L(y))=(\phi_L(y)+1)^{-1} \phi_L(x) (1+\phi_L(y))=\phi_L(x).
\]

Each $a \in A(D)$ has $\phi_L \gamma_D(a)=t_{\kappa(a)}-1$, so if $a$ and $a'$ are arcs from different components of $L$, then $\phi_L \gamma_D(a) \neq \phi_L \gamma_D(a')$. Therefore $\gamma_D(a)$ and $\gamma_D(a')$ fall into different orbits in $Q_A(L)$. As $\sigma:Q(L) \to Q_A(L)$ is a quandle homomorphism with $\sigma(a)=\gamma_D(a)$ and $\sigma(a')=\gamma_D(a')$, it follows that $a$ and $a'$ fall into different orbits in $Q(L)$. \end{proof}

The situation is quite different in the total multivariate Alexander quandle $U(L)$. This quandle has infinitely many orbits, because there there are infinitely many different $\phi_L(x)$ values for $x \in U(L)$, e.g. $-1\pm t_1^n$ for any integer $n$.

The following notion will be useful.

\begin{definition}
\label{wdef}
A \emph{quandle word} is a sequence of symbols, which can be constructed using these three rules. 
\begin{itemize}
    \item Any element $a \in A(D)$ is a quandle word.
    \item If $W_1$ and $W_2$ are quandle words then so is $(W_1) \triangleright (W_2)$.
    \item If $W_1$ and $W_2$ are quandle words then so is $(W_1) \triangleright^{-1} (W_2)$.
\end{itemize}
\end{definition}

Each quandle word is constructed in a unique way from the elements of $A(D)$, using parentheses and the symbols $\triangleright, \triangleright^{-1}$. We use $x(W)$ to denote the element of $Q(L)$ represented by $W$. As the fundamental quandle $Q(L)$ is generated by the elements of $A(D)$, for each $x \in Q(L)$ we may choose a quandle word $W(x)$ such that $x=x(W(x))$, and $W(x)$ is of the shortest possible length. In general this choice is not unique; there may be distinct quandle words of minimum length that equal $x$ in $Q(L)$. However the ``shortest length'' requirement does imply that $W(a)=a$ $\forall a \in A(D)$.

Each quandle word $W$ has a well-defined interpretation as an element $i(W) \in \Lambda_{\mu}^{A(D)}$, obtained as follows. If $W=a \in A(D)$, then $i(W)=a \in \Lambda_{\mu}^{A(D)}$. If $W$ is $W_1 \triangleright W_2$, then 
\[
i(W)=(\phi_L\gamma_Di(W_2)+1) \cdot i(W_1) - \phi_L\gamma_Di(W_1) \cdot i(W_2)
\]
and if $W$ is $W_1 \triangleright^{-1} W_2$, then
\[
i(W)=(\phi_L \gamma_D i(W_2)+1)^{-1} \cdot (i(W_1)+\phi_L\gamma_Di(W_1) \cdot i(W_2)).
\]

As $\Lambda_{\mu}^{A(D)} \subseteq \Lambda_{\mu}^{Q(L)}$, $i(W)$ and $x(W)$ are both images of $W$ in $\Lambda_{\mu}^{Q(L)}$. Our next lemma relates these two images to each other.

\begin{lemma}
\label{lem16}
Let $[S]$ denote the $\Lambda_{\mu}$-submodule of $\Lambda_{\mu}^{Q(L)}$ generated by $S=S_1 \cup S_2$, where
\[
S_1=\{(\phi_L(\sigma(y))+1)x-\phi_L(\sigma(x))y- (x \triangleright y) \mid x,y \in Q(L)\}
\]
\begin{center} and \end{center}
\[
S_2=\{(\phi_L (\sigma (y))+1)^{-1} \cdot (x+\phi_L(\sigma (x)) y)- (x \triangleright^{-1} y) \mid x,y \in Q(L)\}.
\]
Then every quandle word $W$ has $i(W)-x(W) \in [S]$. 
\end{lemma}
\begin{proof} Note first that the definition of $i$ incorporates the formulas for $\triangleright$ and $\triangleright^{-1}$ in $Q_A(L)$, so an induction based on Lemma \ref{lem12} tells us that every quandle word $W$ has $\gamma_D i (W) = \sigma x(W)$.

If $W=a \in A(D)$ then the assertion of the lemma is obvious, as $i(W)- x(W)=a-a=0$. The proof proceeds using induction on the length of $W$. 
 
 If $W=W_1 \triangleright W_2$ and the lemma holds for $W_1$ and $W_2$, then
\begin{align*}
& i(W)-x(W) =(\phi_L \gamma_D i(W_2)+1)i(W_1) - \phi_L \gamma_D i(W_1) i(W_2)-x(W_1 \triangleright W_2)\\
& =(\phi_L\gamma_Di(W_2)+1)i(W_1) - \phi_L\gamma_Di(W_1) i(W_2)\\
& -(\phi_L \sigma x(W_2)+1)x(W_1)+\phi_L \sigma x(W_1)x(W_2)\\
& +(\phi_L \sigma x(W_2)+1)x(W_1)-\phi_L \sigma x(W_1)x(W_2) -x(W_1 \triangleright W_2)\\
& =(\phi_L \sigma x(W_2)+1) \cdot (i(W_1) - x(W_1))\\
& - \phi_L \sigma x(W_1) \cdot (i(W_2)-x(W_2))\\
& +(\phi_L \sigma x (W_2)+1)x(W_1)-\phi_L \sigma x(W_1)x(W_2) -x(W_1 \triangleright W_2) \text{,}
\end{align*}
the sum of three elements of $[S]$.

If $W=W_1 \triangleright^{-1} W_2$ and the lemma holds for $W_1$ and $W_2$, then
\begin{align*}
& i(W)-x(W) =(\phi_L \gamma_D i(W_2)+1)^{-1} \cdot (i(W_1)+\phi_L\gamma_Di(W_1) i(W_2))-x(W)\\
& = (\phi_L \gamma_D i(W_2)+1)^{-1} \cdot (i(W_1)+\phi_L\gamma_Di(W_1) i(W_2))\\
& -(\phi_L \sigma x(W_2)+1)^{-1} \cdot (x(W_1)+\phi_L \sigma x(W_1) x(W_2))\\
& +(\phi_L \sigma x(W_2)+1)^{-1} \cdot (x(W_1)+\phi_L \sigma x(W_1) x(W_2))-x(W_1 \triangleright^{-1} W_2)\\
& =(\phi_L \sigma x (W_2)+1)^{-1} \cdot (i(W_1)-x(W_1))\\
& +(\phi_L \sigma x(W_2)+1)^{-1} \cdot \phi_L \sigma x(W_1) \cdot (i(W_2)- x(W_2))\\
& +(\phi_L \sigma x(W_2)+1)^{-1} \cdot (x(W_1)+\phi_L \sigma x(W_1) x(W_2))-x(W_1 \triangleright^{-1} W_2).
\end{align*}
Again, this is the sum of three elements of $[S]$.
\end{proof}

\begin{corollary}
\label{cor17}
Let $\widehat{\sigma}:\Lambda_{\mu}^{Q(L)} \to M_A(L)$ be the $\Lambda_{\mu}$-linear map with $\widehat{\sigma}(x)=\sigma(x)$ $\forall x \in Q(L)$. Then $\widehat{\sigma}$ is surjective and its kernel is $[S]$, the submodule of $\Lambda_{\mu}^{Q(L)}$ mentioned in Lemma \ref{lem16}.
\end{corollary}
\begin{proof}
 Notice that $\Lambda_{\mu}^{A(D)} \subseteq \Lambda_{\mu}^{Q(L)}$, and $\widehat{\sigma} | \Lambda_{\mu}^{A(D)} = \gamma_D$. As $\gamma_D$ is surjective, it follows that $\widehat{\sigma}$ is surjective too. 
 
Lemma \ref{lem12} states that $\sigma$ is a quandle homomorphism, so if $x,y \in Q(L)$ then 
 \begin{align*}
  & \widehat{\sigma} \big((\phi_L(\sigma(y))+1)x-\phi_L(\sigma(x))y-(x \triangleright y) \big)\\
  & =(\phi_L(\sigma(y))+1)\sigma(x)-\phi_L(\sigma(x))\sigma(y)-\sigma(x \triangleright y)  = \sigma(x) \triangleright \sigma(y)-\sigma(x \triangleright y)=0 \text{,}
\end{align*}
\begin{center} and \end{center}
\begin{align*}
  & \widehat{\sigma} \big((\phi_L (\sigma (y))+1)^{-1} \cdot (x+\phi_L(\sigma (x)) y)- (x \triangleright^{-1} y) \big)\\
  & =(\phi_L (\sigma (y))+1)^{-1} \cdot (\sigma(x)+\phi_L(\sigma (x)) \sigma(y))- \sigma(x \triangleright^{-1} y)\\
  & = \sigma(x) \triangleright^{-1} \sigma(y)- \sigma(x \triangleright^{-1} y)  = 0 \text{.}
\end{align*}
It follows that $S \subseteq \ker \widehat{\sigma}$. To complete the proof, we need to show that $\ker \widehat{\sigma}$ is contained in $[S]$.

Suppose $z \in \ker \widehat{\sigma}$. Then $z$ is an element of $\Lambda_{\mu}^{Q(L)}$, so
\[
z = \sum_{x \in Q(L)} z(x)x = \sum_{x \in Q(L)} z(x)x(W(x))
\]
for some coefficients $z(x) \in \Lambda_{\mu}$. Let 
\[
z_1= \sum_{x \in Q(L)} z(x)i(W(x)) \quad \text{and} \quad z_2= \sum_{x \in Q(L)} z(x)(i(W(x))-x(W(x))).
\]
We claim that $z_1$ and $z_2$ are both elements of $[S]$.

Notice that if $c$ is a crossing of $D$ with arcs indexed as in Fig.\ \ref{crossfig}, then 
\[
 \rho_D(c)=(1-t_{\kappa(a_2)})a_1+t_{\kappa(a_1)}a_2-a_3
 \]
 \[
=-\phi_L \sigma(a_2)a_1+(\phi_L\sigma(a_1)+1) a_2-(a_2 \triangleright a_1)
\]
is an element of $S$. It follows that $[S]$ contains $\rho_D(\Lambda_{\mu}^{C(D)}) = \ker \gamma_D$. As
\[
\gamma_D (z_1) = \sum_{x \in Q(L)} z(x) \gamma_D i(W(x)) = \sum_{x \in Q(L)} z(x) \sigma x(W(x))
\]
\[
= \widehat{\sigma} \left(\sum_{x \in Q(L)} z(x) x(W(x)) \right) = \widehat{\sigma}(z) =0 \text{,}
\]
we deduce the claim that $z_1 \in [S]$.

The claim that $z_2 \in [S]$ follows from Lemma \ref{lem16}, which tells us that $i(W(x))-x(W(x)) \in [S]$ for every $x \in Q(L)$. 

The two claims imply that $z=z_1-z_2 \in [S]$. \end{proof}

Summarizing the discussion above, we have the following.

\begin{theorem}
\label{qpres}
Let $L=K_1 \cup \dots \cup K_{\mu}$ be a link with fundamental quandle $Q(L)$. Then there are surjective functions $\kappa:Q(L) \to \{1, \dots, \mu\}$ and $\widehat{\sigma}:\Lambda_{\mu}^{Q(L)} \to M_A(L)$ with the following properties.
\begin{enumerate}
    \item The orbits of $Q(L)$ are the sets $\kappa^{-1}(\{1\}),\dots,\kappa^{-1}(\{\mu\})$.
    \item The restriction $\sigma=\widehat{\sigma}|Q(L)$ is a quandle homomorphism onto $Q_A(L)$.
    \item For each $x \in Q(L)$, $\phi_L  \sigma(x)=t_{\kappa(x)}-1$.
    \item The map $\widehat{\sigma}$ is $\Lambda_{\mu}$-linear, and its kernel is the submodule of $\Lambda_{\mu}^{Q(L)}$ generated by $S=S_1 \cup S_2$, where
    \[
S_1=\{t_{\kappa(y)}x+(1-t_{\kappa(x)})y- (x \triangleright y) \mid x,y \in Q(L)\}
\]
\center{and} 
\[
S_2=\{t_{\kappa(y)}^{-1}x + t_{\kappa(y)}^{-1} (t_{\kappa(x)}-1) y- (x \triangleright^{-1} y) \mid x,y \in Q(L)\}.
\]
\end{enumerate}
\end{theorem}
\begin{proof}
Parts 1 and 2 follow directly from Lemmas \ref{lem12} and \ref{lem14}. Part 3 is justified using ideas from the proof of Lemma \ref{lem14}: if $D$ is a diagram of $L$ and $x \in Q(L)$, then there is an $a \in A(D)$ that lies in the same orbit of $Q(L)$ as $x$; then $\phi_L \sigma(x)=\phi_L \sigma(a)=\phi_L \gamma_D(a)=t_{\kappa(a)}-1$. The descriptions of the elements of $S_1$ and $S_2$ given in part 4 are obtained from the descriptions in Lemma \ref{lem16} using the formula of part 3.
\end{proof}

The arguments in this section remain valid if they are modified by replacing everything in $Q(L)$ with its image under $\sigma$ in $Q_A(L)$. For instance the map $\sigma:Q(L) \to Q_A(L)$ is replaced by the identity map $\textup{id}:Q_A(L) \to Q_A(L)$, and if $a \in A(D)$ appears in a quandle word $W$ then $\sigma x(W)$, the element of $Q_A(L)$ represented by $W$, involves $\sigma(a)=\gamma_D(a)$ rather than $a$ itself. We provide the modified statements of Lemma \ref{lem16}, Corollary \ref{cor17} and Theorem \ref{qpres} below, and leave it to the reader to adjust the details of the arguments.

\begin{lemma}
\label{lem19}
Let $[S_A]$ denote the $\Lambda_{\mu}$-submodule of $\Lambda_{\mu}^{Q_A(L)}$ generated by $S_A=S_{A1} \cup S_{A2}$, where
\[
S_{A1}=\{(\phi_L(y)+1)x-\phi_L(x)y- (x \triangleright y) \mid x,y \in Q_A(L)\}
\]
\begin{center} and \end{center}
\[
S_{A2}=\{(\phi_L (y)+1)^{-1} \cdot (x+\phi_L(x) y)- (x \triangleright^{-1} y) \mid x,y \in Q_A(L)\}.
\]
Then every quandle word $W$ has $\sigma i(W)-\sigma x(W) \in [S_A]$. 
\end{lemma}

\begin{corollary}
\label{cor20}
Let $\widehat{\textup{id}}:\Lambda_{\mu}^{Q_A(L)} \to M_A(L)$ be the $\Lambda_{\mu}$-linear map with $\widehat{\textup{id}}(x)=x$ $\forall x \in Q_A(L)$. Then $\widehat{\textup{id}}$ is surjective and its kernel is $[S_A]$.
\end{corollary}

\begin{theorem}
\label{aqpres}
Let $L=K_1 \cup \dots \cup K_{\mu}$ be a link with fundamental multivariate Alexander quandle $Q_A(L)$. Then there are surjective functions $\kappa:Q_A(L) \to \{1, \dots, \mu\}$ and $\widehat{\textup{id}}:\Lambda_{\mu}^{Q_A(L)} \to M_A(L)$ with the following properties.
\begin{enumerate}
    \item The orbits of $Q_A(L)$ are the sets $\kappa^{-1}(\{1\}),\dots,\kappa^{-1}(\{\mu\})$.
    \item The restriction $\textup{id}=\widehat{\textup{id}}|Q_A(L)$ is the identity map of $Q_A(L)$.
    \item For each $x \in Q_A(L)$, $\phi_L(x)=t_{\kappa(x)}-1$.
    \item The map $\widehat{\textup{id}}$ is $\Lambda_{\mu}$-linear, and its kernel is the submodule of $\Lambda_{\mu}^{Q_A(L)}$ generated by $S_A=S_{A1} \cup S_{A2}$, where
    \[
S_{A1}=\{t_{\kappa(y)}x+(1-t_{\kappa(x)})y- (x \triangleright y) \mid x,y \in Q_A(L)\}
\]
\center{and} 
\[
S_{A2}=\{t_{\kappa(y)}^{-1}x + t_{\kappa(y)}^{-1} (t_{\kappa(x)}-1) y- (x \triangleright^{-1} y) \mid x,y \in Q_A(L)\}.
\]
\end{enumerate}
\end{theorem}

\section{Proof of Theorem \ref{main}}
\label{secproof}
In this section we use the module presentations of Sec.\ \ref{seclemm} to verify the three implications of Theorem \ref{main}. Suppose $L=K_1 \cup \dots \cup K_{\mu}$ and $L'=K'_1 \cup \dots \cup K'_{\mu'}$ are links.

$(1 \implies 2)$: Let $f:Q(L) \to Q(L')$ be a quandle isomorphism. Let $\widehat{f}:\Lambda_{\mu}^{Q(L)} \to \Lambda_{\mu}^{Q(L')}$ be the isomorphism defined by $f$. Also, let $\kappa:Q(L) \to \{1, \dots, \mu\}$, $\kappa':Q(L') \to \{1, \dots, \mu' \}$, $\widehat{\sigma}:\Lambda_{\mu}^{Q(L)} \to M_A(L)$ and $\widehat{\sigma}':\Lambda_{\mu}^{Q(L')} \to M_A(L')$ be the surjections mentioned in Theorem \ref{qpres}.

The orbits of $Q(L)$ and $Q(L')$ are matched to each other by the isomorphism $f$, of course. It follows that $\mu = \mu'$, and we may reindex the components of $L=K_1 \cup \dots \cup K_{\mu}$ so that $\kappa (x)=\kappa'(f(x))$ $\forall x \in Q(L)$. Then part 4 of Theorem \ref{qpres} guarantees that $\widehat{f}(\ker \widehat{\sigma}) = \ker \widehat{\sigma}'$. Hence $\widehat{f}$ defines an isomorphism $\widetilde{f}:M_A(L) \to M_A(L')$ of $\Lambda_{\mu}$-modules.

Suppose $y \in Q_A(L)$; then $y=\widehat{\sigma}(x)$ for some $x \in Q(L)$. As $\widetilde{f}$ is defined by $\widehat{f}$, and $\widehat{f}$ in turn is defined by $f$, $\widetilde{f}(y)=\widetilde{f}(\widehat{\sigma}(x))=\widehat{\sigma}'(f(x))$ is an element of $\widehat{\sigma}'(Q(L'))=Q_A(L')$. Conversely, if $y' \in Q_A(L')$ then $y'=\widehat{\sigma}'(x')$ for some $x' \in Q(L')$. Then $y'=\widehat{\sigma}'(f f^{-1}(x'))=\widetilde{f}(\widehat{\sigma} f^{-1}(x'))$ is an element of $\widetilde{f}(Q_A(L))$. We conclude that $\widetilde{f}(Q_A(L))=Q_A(L')$. As $\widetilde{f}$ is $\Lambda_{\mu}$-linear and the quandle operations of $Q_A(L)$ and $Q_A(L')$ are defined by the formulas of Propositions \ref{multiq1} and \ref{multiq2}, $\widetilde{f}$ restricts to a quandle isomorphism $Q_A(L) \to Q_A(L')$. 

$(2 \implies 3)$: Suppose there is a quandle isomorphism $f:Q_A(L) \to Q_A(L')$. Let $\widehat{f}:\Lambda_{\mu}^{Q_A(L)} \to \Lambda_{\mu}^{Q_A(L')}$ be the isomorphism defined by $f$, and let $\kappa:Q_A(L) \to \{1, \dots, \mu\}$, $\kappa':Q_A(L') \to \{1, \dots, \mu' \}$, $\widehat{\textup{id}}:\Lambda_{\mu}^{Q_A(L)} \to M_A(L)$ and $\widehat{\textup{id}}':\Lambda_{\mu}^{Q_A(L')} \to M_A(L')$ be the surjections mentioned in Theorem \ref{aqpres}.

The isomorphism $f$ matches the orbits of $Q_A(L)$ to the orbits of $Q_A(L')$, so $\mu = \mu'$. If necessary, we may reindex the components of $L=K_1 \cup \dots \cup K_{\mu}$ so that $\kappa (x)=\kappa'(f(x))$ $\forall x \in Q_A(L)$. Then part 4 of Theorem \ref{aqpres} guarantees that $\widehat{f}(\ker \widehat{\textup{id}}) = \ker \widehat{\textup{id}}'$. Therefore $\widehat{f}$ defines an isomorphism $\widetilde{f}:M_A(L) \to M_A(L')$ of $\Lambda_{\mu}$-modules.

As $\kappa (x)=\kappa'(f(x))$ $\forall x \in Q_A(L)$, each $x \in Q_A(L)$ has $\phi_L(x)=t_{\kappa(x)}-1=t_{\kappa'(\widetilde{f}(x))}-1=\phi_{L'}(\widetilde{f}(x))$. The module $M_A(L)$ is generated by $Q_A(L)$, so it follows that $\phi_{L'} \widetilde{f} =\phi_L: M_A(L) \to I_{\mu}$. 

$(3 \implies 4)$: If $f:M_A(L) \to M_A(L')$ is an isomorphism with $\phi_{L'} f = \phi_L$, then $f$ maps $U(L)$ onto $U(L')$. As $f$ is $\Lambda_{\mu}$-linear and the quandle operations of $U(L)$ and $U(L')$ are given by the formulas of Propositions \ref{multiq1} and \ref{multiq2}, it follows that $f$ defines a quandle isomorphism $U(L) \to U(L')$.

\section{Counterexamples}
\label{seccoun}

In this section we provide counterexamples for the converses of the implications $1 \implies 2$ and $3 \implies 4$ of Theorem \ref{main}.

Let $K$ be a knot with Alexander polynomial $1$. Two famous examples of such knots are the Conway and Kinoshita-Terasaka knots; there are infinitely many others (cf.\ \cite[p.\ 167]{Ro}, for instance). Crowell \cite{C2a} proved that for such a knot, the map $\phi_K:M_A(K) \to I_1$ is an isomorphism. 

If $D$ is any diagram of $K$ then every $a \in A(D)$ has $\phi_K \gamma_D(a) = t_1-1$. As $\phi_K$ is an isomorphism, it follows that every $a \in A(D)$ has the same image under $\gamma_D$. Therefore $Q_A(K)$ has only one element. 

We conclude that all knots with Alexander polynomial $1$ have the same fundamental multivariate Alexander quandle, up to isomorphism. These knots are distinguished from each other by their fundamental quandles, of course \cite{J, M}, so they demonstrate that condition 2 of Theorem \ref{main} does not imply condition 1.

 Now, let $L_1$ and $L_2$ be two-component links with homeomorphic exteriors and different Alexander polynomials. There are infinitely many such pairs, but we focus on the pair discussed by Rolfsen \cite[ps.\ 195 and 196]{Ro}, which was also mentioned by Joyce \cite{J}. With one choice of component indices, the Alexander polynomials of $L_1$ and $L_2$ are $1+t_1t_2$ and $1+t_1^3 t_2$ ($1+xy$ and $1+x^3y$, in Rolfsen's notation). The difference between the Alexander polynomials guarantees that the Alexander modules of $L_1$ and $L_2$ are not isomorphic $\Lambda_{2}$-modules, even if the component indices in one link are reversed; it follows that $L_1$ and $L_2$ do not satisfy condition 3 of Theorem \ref{main}. 
 
A homeomorphism between the link exteriors induces an isomorphism $f:G_1 \to G_2$ between the groups of $L_1$ and $L_2$, and this $f$ induces an isomorphism between the integral group rings of the abelianizations of $G_1$ and $G_2$. If both integral group rings are identified with $\Lambda_2$ in the usual way, then this induced isomorphism becomes an automorphism $g$ of $\Lambda_2$ that is not the identity map; with one choice of component indices, the automorphism has $g(t_1)=t_1$ and $g(t_2)=t_1^2t_2$. The link module sequences of $L_1$ and $L_2$ are semilinearly isomorphic, with respect to this automorphism $g$. That is to say, there is a map $g':M_A(L_1) \to M_A(L_2)$ which is an isomorphism of abelian groups, such that $g'(\lambda x)=g(\lambda)g'(x)$ and $\phi_{L_2} g'(x)=g \phi_{L_1}(x)$ $\forall \lambda \in \Lambda_2$ $\forall x \in M_A(L_1)$.

If $x \in U(L_1)$ then $\phi_{L_1}(x)+1= \pm t_1^mt_2^n$ for some integers $m$ and $n$. Therefore $\phi_{L_2}g'(x)+1=g\phi_{L_1}(x)+1= \pm t_1^{m+2n}t_2^n$, so $g'(x) \in U(L_2)$. Conversely, if $g'(x) \in U(L_2)$ then $\phi_{L_2}g'(x)+1 = \pm t_1^mt_2^n$ for some integers $m$ and $n$, so $g\phi_{L_1}(x)+1 = \pm t_1^mt_2^n =g(\pm t_1^{m-2n}t_2^n)$. As $g$ is an automorphism of $\Lambda_2$, it follows that $\phi_{L_1}(x)+1 = \pm t_1^{m-2n}t_2^n$, so $x \in U(L_1)$. We conclude that $g'$ maps $U(L_1)$ onto $U(L_2)$. Using the formulas for $\triangleright$ and $\triangleright^{-1}$ in Propositions \ref{multiq1} and \ref{multiq2}, it follows that $g'$ defines an isomorphism between the quandles $U(L_1)$ and $U(L_2)$. Thus $L_1$ and $L_2$ satisfy condition 4 of Theorem \ref{main}, but not condition 3.

\section{Proposition \ref{red}}
\label{secred}

Proving the first part of Proposition \ref{red} is not difficult. If $L$ and $L'$ have equivalent link module sequences then there is an isomorphism $f:M_A(L) \to M_A(L')$ with $\phi_L=\phi_{L'}f$. An isomorphism $f^{red}:M_A^{red}(L) \to M_A^{red}(L')$ with $\phi_L^{red}=\phi_{L'}^{red} f^{red}$ is obtained directly from $f$, by setting all $t_i=t$ in matrices representing $f$ and the $\rho,\gamma$ maps of Definition \ref{almod}. (Equivalently, the functor $\Lambda \otimes_{\Lambda_{\mu}}-$ is applied to the whole situation.)

The second part of Proposition \ref{red} is verified by exhibiting two links with inequivalent link module sequences, whose reduced link module sequences are equivalent. We do not know of a computer program, reference or web site for the link module sequences of individual links, so we provide a detailed account.

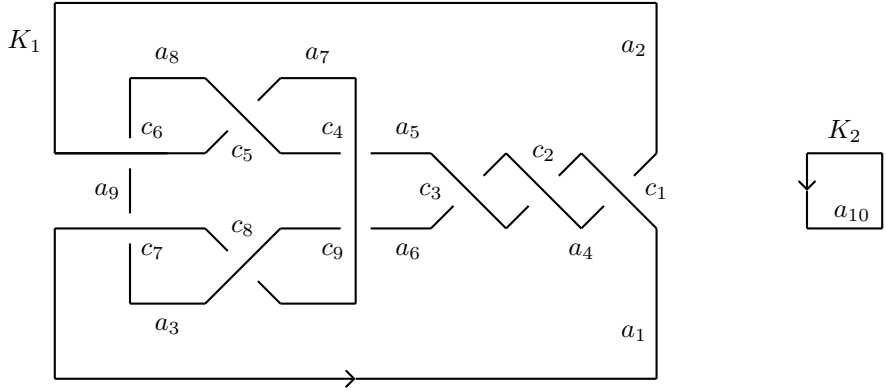
\begin{figure} [bht]
\centering
\begin{tikzpicture} 
\draw [->] [>=angle 90] [thick] (0,0) -- (4,0);
\draw [thick] (10,2) -- (10,2.5);
\draw [thick] (10,2) -- (11,2);
\draw [->] [>=angle 90] [thick] (10,3) -- (10,2.5);
\draw [thick] (11,3) -- (10,3);
\draw [thick] (11,3) -- (11,2);
\draw [thick] (8,0) -- (4,0);
\draw [thick] (0,5) -- (8,5);
\draw [thick] (0,5) -- (0,3);
\draw [thick] (0,0) -- (0,2);
\draw [thick] (2,2) -- (0,2);
\draw [thick] (1,2.2) -- (1,2.8);
\draw [thick] (2,3) -- (0,3);
\draw [thick] (1,1) -- (1,1.8);
\draw [thick] (1,3.2) -- (1,4);
\draw [thick] (4,1) -- (4,4);
\draw [thick] (3,4) -- (4,4);
\draw [thick] (3,1) -- (4,1);
\draw [thick] (1.5,3) -- (0,3);
\draw [thick] (2,3) -- (2.3,3.3);
\draw [thick] (2.7,3.7) -- (3,4);
\draw [thick] (2,2) -- (2.3,1.7);
\draw [thick] (3,1) -- (2.7,1.3);
\draw [thick] (2,1) -- (3,2);
\draw [thick] (2,1) -- (1,1);
\draw [thick] (3,2) -- (3.8,2);
\draw [thick] (5,2) -- (4.2,2);
\draw [thick] (2,4) -- (1,4);
\draw [thick] (2,4) -- (3,3);
\draw [thick] (3.8,3) -- (3,3);
\draw [thick] (4.2,3) -- (5,3);
\draw [thick] (5,3) -- (6,2);
\draw [thick] (6,3) -- (7,2);
\draw [thick] (7,3) -- (8,2);
\draw [thick] (6,2) -- (6.3,2.3);
\draw [thick] (7,3) -- (6.7,2.7);
\draw [thick] (5,2) -- (5.3,2.3);
\draw [thick] (6,3) -- (5.7,2.7);
\draw [thick] (7,2) -- (7.3,2.3);
\draw [thick] (8,3) -- (7.7,2.7);
\draw [thick] (8,3) -- (8,5);
\draw [thick] (8,2) -- (8,0);
\node at (7.7,0.6) {$a_1$};
\node at (7.7,4.4) {$a_2$};
\node at (7,1.7) {$a_4$};
\node at (4.7,3.3) {$a_5$};
\node at (4.7,1.7) {$a_6$};
\node at (3.5,4.3) {$a_7$};
\node at (1.5,4.3) {$a_8$};
\node at (1.5,0.7) {$a_3$};
\node at (0.7,2.5) {$a_9$};
\node at (10.6,2.2) {$a_{10}$};
\node at (8,2.5) {$c_1$};
\node at (6.5,3) {$c_2$};
\node at (5,2.5) {$c_3$};
\node at (3.7,3.3) {$c_4$};
\node at (2.5,3) {$c_5$};
\node at (1.3,3.3) {$c_6$};
\node at (1.3,1.7) {$c_7$};
\node at (2.5,2) {$c_8$};
\node at (3.7,1.7) {$c_9$};
\node at (-0.4,4.5) {$K_1$};
\node at (10.5,3.3) {$K_2$};
\end{tikzpicture}
\caption{The link $L$.}
\label{knotsfig}
\end{figure}

Let $L$ be the split union of the knot $9_{46}$ and an unknot, with arcs and crossings in the diagram $D$ indexed as in Fig.\ \ref{knotsfig}. We analyze the link module sequence of $L$ by simplifying the Alexander module presentation given in Definition \ref{almod}. 

First, we eliminate six of the ten generators as follows.
\begin{align*}
    & \text{using }\rho_D(c_7): t_1 \gamma_D(a_9)= (t_1-1)\gamma_D(a_1) + \gamma_D(a_3)\\
    & \text{using }\rho_D(c_1): \gamma_D(a_4)=(1-t_1)\gamma_D(a_1) + t_1 \gamma_D(a_2)\\
    & \text{using }\rho_D(c_6): \gamma_D(a_8)=t_1 \gamma_D(a_9)+(1-t_1)\gamma_D(a_2)\\
    & \qquad = (t_1-1)\gamma_D(a_1) +(1-t_1)\gamma_D(a_2)+\gamma_D(a_3)
    \end{align*}
    \begin{align*}
    & \text{using }\rho_D(c_5): t_1 \gamma_D(a_7)=\gamma_D(a_2)+(t_1-1)\gamma_D(a_8) \\
    & \qquad = (t_1^2-2t_1 +1)\gamma_D(a_1) +(2t_1 - t_1^2)\gamma_D(a_2)+(t_1-1)\gamma_D(a_3)\\
    & \text{using }\rho_D(c_4):  \gamma_D(a_5)=(1-t_1)\gamma_D(a_7)+t_1\gamma_D(a_8)\\
    & \qquad =(t_1^{-1}-1)t_1\gamma_D(a_7)+t_1\gamma_D(a_8) \\
    & \qquad = (t_1^{-1}-3+2t_1)\gamma_D(a_1) +(2-2t_1)\gamma_D(a_2)+(2-t_1^{-1})\gamma_D(a_3)\\
    & \text{using }\rho_D(c_9): \gamma_D(a_6)=(1-t_1)\gamma_D(a_7)+t_1\gamma_D(a_3)\\
    & \qquad = (t_1^{-1}-1)t_1\gamma_D(a_7)+t_1\gamma_D(a_3)\\
    & \qquad = (t_1^{-1}-3+3t_1-t_1^2)\gamma_D(a_1)+(2-3t_1+t_1^2)\gamma_D(a_2)+(2-t_1^{-1})\gamma_D(a_3)
\end{align*}

We conclude that $M_A(L)$ is generated by $\gamma_D(a_1),\gamma_D(a_2),\gamma_D(a_3)$ and $\gamma_D(a_{10})$, subject to the relations $0=\gamma_D \rho_D(c_2),0=\gamma_D \rho_D(c_3)$ and $0=\gamma_D \rho_D(c_8)$:
\begin{align*}
    & 0=\gamma_D \rho_D(c_2) =(1-t_1) \gamma_D(a_4)+ t_1\gamma_D(a_5) - \gamma_D(a_1)\\
    & \qquad = (1-5t_1+3t_1^2)\gamma_D(a_1)+(3t_1-3t_1^2)\gamma_D(a_2)+(2t_1-1)\gamma_D(a_3)\\
    & 0=\gamma_D \rho_D(c_3) =(1-t_1) \gamma_D(a_5)+ t_1\gamma_D(a_4) - \gamma_D(a_6)\\
    & \qquad = (-1+3t_1-2t_1^2)\gamma_D(a_1)+(2t_1^2-t_1)\gamma_D(a_2)+(-2t_1+1)\gamma_D(a_3) \\
    & 0=\gamma_D \rho_D(c_8) =(1-t_1) \gamma_D(a_3)+ t_1\gamma_D(a_7) - \gamma_D(a_1)\\
    & \qquad = (t_1^2-2t_1)\gamma_D(a_1)+(2t_1-t_1^2)\gamma_D(a_2)
\end{align*}
Rewriting the presentation in terms of the generators  $\gamma_D(a_1),x=\gamma_D(a_2)-\gamma_D(a_1),y=\gamma_D(a_3)-t_1\gamma_D(a_2)+(t_1-1)\gamma_D(a_1)$ and $\gamma_D(a_{10})$, we obtain the following relations. 
\begin{align*}
    & 0= (2t_1-t_1^2)x-(-2t_1+1)y\\
    & 0=(-2t_1+1)y \\
    & 0=(2t_1-t_1^2)x
\end{align*}
The first relation is redundant, so
\begin{equation}
\label{module1}
M_A(L)\cong \Lambda_2 \oplus (\Lambda_2/(2t_1-t_1^2)) \oplus (\Lambda_2/(2t_1-1)) \oplus \Lambda_2 \text{,}
\end{equation}
with the four summands generated by $\gamma_D(a_1),x,y$ and $\gamma_D(a_{10})$ respectively. The map $\phi_L$ is given by $\phi_L(\gamma_D(a_1))=t_1-1$, $\phi_L(x)=\phi_L(\gamma_D(a_2)-\gamma_D(a_1))=t_1-1-(t_1-1)=0$, $\phi_L(y)=\phi_L(\gamma_D(a_3)-t_1\gamma_D(a_2)+(t_1-1)\gamma_D(a_1))=t_1-1-t_1(t_1-1)+(t_1-1)^2=0$ and $\phi_L(\gamma_D(a_1))=t_2-1$.

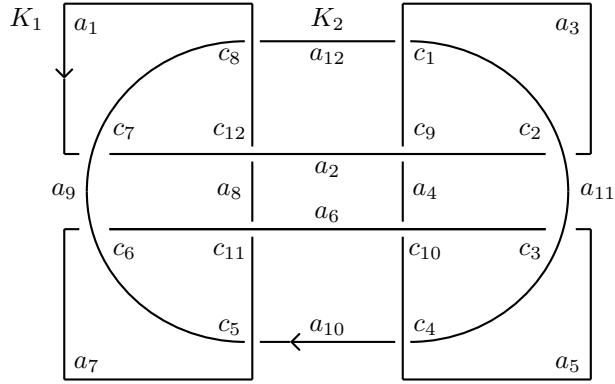
\begin{figure} 
\centering
\begin{tikzpicture} 
\draw [thick] (3,2.1) -- (3,2.9);
\draw [thick] (5,2.1) -- (5,2.9);
\draw [thick] (5,3.1) -- (5,5);
\draw [thick] (7.5,5) -- (5,5);
\draw [thick] (7.5,5) -- (7.5,3);
\draw [thick] (7.3,3) -- (7.5,3);
\draw [thick] (3,3.1) -- (3,5);
\draw [thick] (4.9,4.5) -- (3.1,4.5);
\draw [->] [>=angle 90] [thick] (4.9,.5) -- (3.5,.5);
\draw [thick] (3.1,.5) -- (3.5,.5);
\draw [thick] (5,0) -- (5,1.9);
\draw [thick] (3,0) -- (3,1.9);
\draw [thick] (1.1,2) -- (6.9,2);
\draw [thick] (1.1,3) -- (6.9,3);
\draw [thick] (7.3,2) -- (7.5,2);
\draw [thick] (7.5,2) -- (7.5,0);
\draw [thick] (5,0) -- (7.5,0);
\draw [thick] (1.1,2) -- (6.9,2);
\draw [thick] (0.5,5) -- (3,5);
\draw [->] [>=angle 90] [thick] (0.5,5) -- (0.5,4);
\draw [thick] (0.5,3) -- (0.5,4);
\draw [thick] (.7,3) -- (0.5,3);
\draw [thick] (0.5,0) -- (3,0);
\draw [thick] (0.5,0) -- (0.5,2);
\draw [thick] (.7,2) -- (0.5,2);
\draw [thick] (5.1,0.5) to [out=0, in=-90] (7.2,2.5);
\draw [thick] (7.2,2.5) to [out=90, in=0] (5.1,4.5);
\draw [thick] (2.9,.5) to [out=180, in=-90] (0.8,2.5);
\draw [thick] (2.9,4.5) to [out=180, in=90] (0.8,2.5);
%\draw [thick] (7.2,3) to [out=0, in=90] (5,4.2);
\node at (0.8,4.7) {$a_1$};
\node at (4,2.8) {$a_2$};
\node at (7.2,4.7) {$a_3$};
\node at (5.3,2.5) {$a_4$};
\node at (7.2,0.2) {$a_5$};
\node at (4,2.2) {$a_6$};
\node at (0.8,0.2) {$a_7$};
\node at (2.7,2.5) {$a_8$};
\node at (0.5,2.5) {$a_9$};
\node at (4,.7) {$a_{10}$};
\node at (7.6,2.5) {$a_{11}$};
\node at (4,4.3) {$a_{12}$};
\node at (5.3,4.3) {$c_1$};
\node at (6.7,3.3) {$c_2$};
\node at (6.7,1.7) {$c_3$};
\node at (5.3,0.7) {$c_4$};
\node at (2.7,0.7) {$c_5$};
\node at (1.3,1.7) {$c_6$};
\node at (1.3,3.3) {$c_7$};
\node at (2.7,4.3) {$c_8$};
\node at (5.3,3.3) {$c_9$};
\node at (5.3,1.7) {$c_{10}$};
\node at (2.7,1.7) {$c_{11}$};
\node at (2.7,3.3) {$c_{12}$};
\node at (0,4.8) {$K_1$};
\node at (4,4.8) {$K_2$};
\end{tikzpicture}
\caption{Milnor's link, $M$.}
\label{milsfig}
\end{figure}

Now, let $M$ be the link with the diagram $E$ pictured in Fig.\ \ref{milsfig}. This link was mentioned by Milnor \cite{Mi}; a distinctive property is that all of its $\bar \mu$-invariants equal $0$.

Again, we analyze the Alexander module by simplifying the presentation from Definition \ref{almod}. First, we eliminate six of the generators.
\[
    \text{using }\rho_E(c_3): \gamma_E(a_5)=(1-t_1)\gamma_E(a_{11})+t_2\gamma_E(a_6) \]\[
    \text{using }\rho_E(c_5): t_1 \gamma_E(a_{10})= (t_2-1)\gamma_E(a_7) + \gamma_E(a_9)\]\[
    \text{using }\rho_E(c_6): \gamma_E(a_7)=(1-t_1)\gamma_E(a_9)+t_2\gamma_E(a_6) \]\[
    \text{using }\rho_E(c_8): t_1 \gamma_E(a_{12})=(t_2-1)\gamma_E(a_1) + \gamma_E(a_9)\]\[
    \text{using }\rho_E(c_{10}): t_1 \gamma_E(a_4)= (t_1-1)\gamma_E(a_6)+\gamma_E(a_5)\]\[
    \text{using }\rho_E(c_{12}): t_1 \gamma_E(a_8)=(t_1-1)\gamma_E(a_2)+\gamma_E(a_1)
    \]
    The generator $\gamma_E(a_3)$ is eliminated next.
    \begin{align*}
    & \text{using }\rho_E(c_9): \gamma_E(a_3)=(1-t_1)\gamma_E(a_2) + t_1\gamma_E(a_4)\\
    & =(1-t_1)\gamma_E(a_2) + (t_1-1)\gamma_E(a_6)+\gamma_E(a_5)\\
    & = (1-t_1)\gamma_E(a_2) + (t_1-1+t_2)\gamma_E(a_6)+(1-t_1)\gamma_E(a_{11})
    \end{align*}
    The last generator we eliminate is $\gamma_E(a_1)$.
    \begin{align*}
    & \text{using }\rho_E(c_{11}): \gamma_E(a_1)=\gamma_E(a_1)-0=\gamma_E(a_1)-\gamma_E \rho_E(c_{11})\\
    & =\gamma_E(a_1) - (1-t_1)\gamma_E(a_6)-t_1 \gamma_E(a_8)+\gamma_E(a_7)\\
    & =\gamma_E(a_1) - (1-t_1)\gamma_E(a_6) +(1-t_1)\gamma_E(a_2)-\gamma_E(a_1)+\gamma_E(a_7)\\
    &= (t_1-1+t_2)\gamma_E(a_6)+(1-t_1)\gamma_E(a_2)+(1-t_1)\gamma_E(a_9)
    \end{align*}
    The Alexander module $M_A(M)$ is defined by the four remaining generators, $\gamma_E(a_2),\gamma_E(a_6),\gamma_E(a_9)$ and $\gamma_E(a_{11})$, and the four remaining crossing relations:
    \begin{align*}
        & 0= \gamma_E \rho_E(c_1)=(1-t_2)\gamma_E(a_3)+t_1 \gamma_E(a_{12})- \gamma_E(a_{11})\\
        & = (1-t_2)((1-t_1)\gamma_E(a_2) + (t_1-1+t_2)\gamma_E(a_6)+(1-t_1)\gamma_E(a_{11}))\\
        & + (t_2-1)\gamma_E(a_1) + \gamma_E(a_9)- \gamma_E(a_{11})\\
        & = (-t_1 t_2 + t_1 + t_2)\gamma_E(a_9)+(t_1 t_2 - t_1 - t_2)\gamma_E(a_{11}) \\
        & 0= \gamma_E \rho_E(c_2)=(1-t_1)\gamma_E(a_{11})+t_2 \gamma_E(a_{2})- \gamma_E(a_3)\\
        & = (t_1+t_2-1)\gamma_E(a_{2}) - (t_1+t_2-1)\gamma_E(a_6)
        \end{align*}
        \begin{align*}
        & 0= \gamma_E \rho_E(c_4)=(1-t_2)\gamma_E(a_5)+t_1 \gamma_E(a_{10})- \gamma_E(a_{11})\\
        & = (-t_1 t_2 + t_1 + t_2)\gamma_E(a_9)+(t_1 t_2 - t_1 - t_2)\gamma_E(a_{11})\\
        & 0= \gamma_E \rho_E(c_7)=(1-t_1)\gamma_E(a_9)+t_2 \gamma_E(a_2)- \gamma_E(a_1)\\
        & = (t_1+t_2-1)\gamma_E(a_2)-(t_1+t_2-1)\gamma_E(a_6)
    \end{align*}
The third and fourth relations are redundant, so $M_A(M)$ is generated by four elements,  $\gamma_E(a_2),v=\gamma_E(a_9)-\gamma_E(a_{11}),w=\gamma_E(a_2)-\gamma_E(a_6)$ and $\gamma_E(a_9)$, subject to the relations $(t_1 + t_2 -t_1 t_2)v=0$ and $(t_1 + t_2 - 1)w=0$. That is, \begin{equation}
\label{module2}
M_A(M)\cong \Lambda_2 \oplus (\Lambda_2/(t_1 + t_2 -t_1 t_2)) \oplus (\Lambda_2/(t_1 + t_2 - 1)) \oplus \Lambda_2 \text{,}
\end{equation}
with the four summands generated by $\gamma_E(a_2),v,w$ and $\gamma_E(a_{10})$ respectively. Comparing (\ref{module1}) and (\ref{module2}), we see that $M_A(L)$ and $M_A(M)$ are non-isomorphic $\Lambda_2$-modules, and they remain non-isomorphic if the component indices in either link are reversed; but $M_A^{red}(L)$ and $M_A^{red}(M)$ are isomorphic $\Lambda$-modules. Moreover, $\phi_M(\gamma_E(a_2))=t_1-1, \phi_M(v)=0=\phi_M(w)$ and $\phi_M(\gamma_E(a_9))=t_2-1$, so the obvious isomorphism $f:M_A^{red}(L) \to M_A^{red}(M)$ has $\phi_L^{red}=\phi_M^{red} f$.

\section{Conclusion}
\label{secprob}

Here are five comments about the ideas we have presented.

1. After completing the present paper, we were able to verify that $3 \centernot \implies 2$ in Theorem \ref{main} \cite{mvaq3}. That is, the classical theory involving the Alexander module and link module sequence is genuinely improved by incorporating $Q_A(L)$. 

2. The multivariate Alexander modules of links have been thoroughly studied, but the theory is far from completely settled. For instance, it does not seem to be known whether it is possible for two links to have isomorphic Alexander modules but inequivalent link module sequences; and it does not seem to be known whether the elementary ideals of $M_A(L)$ are stronger or weaker invariants than the elementary ideals of $\ker \phi_L$. Perhaps quandle ideas can be used to address some of these outstanding problems.

3. The multivariate Alexander module is connected to many other link invariants, including the Arf invariant, elementary ideals, linking numbers, Milnor's $\bar \mu$-invariants and signatures. (An invaluable survey of many of these connections is provided by Hillman \cite{H}.) Theorem \ref{main} suggests that the connections between multivariate Alexander modules and other link invariants should extend to $Q_A(L)$ and (hence) $Q(L)$. As far as we know, the only such connection that has been explicitly described in the literature is the connection between linking numbers and quandle colorings, discussed by Harrell and Nelson \cite{HN}. 

4. $U(L)$ is much larger than $Q_A(L)$, and has many other subquandles. How are these subquandles related to each other? What do they tell us about $L$?

5. The fundamental quandles of links are complicated structures, difficult to investigate directly. A valuable strategy for dealing with this difficulty is to consider link colorings with values in smaller quandles, rather than working with fundamental quandles themselves. In the case of $Q_A(L)$, some of these quandle colorings have already been described indirectly, as colorings with domain $M_A(L)$ and values in smaller $\Lambda_{\mu}$-modules \cite{Tcol}. Perhaps these module-based colorings will be useful in addressing the issues mentioned in the first four comments.

\textbf{Acknowledgment}. We are grateful to J. A. Hillman for his informative correspondence regarding Alexander modules and link module sequences.

\end{document}